\documentclass[11pt,a4paper,twoside]{article}

\usepackage{amssymb,amscd,amsfonts,amsbsy}
\usepackage{enumerate}
\usepackage{amsmath,amsthm}
\usepackage{mathrsfs}
\usepackage{epsf,epsfig}
\usepackage{pdfsync}
\usepackage[all]{xy}

\usepackage[colorlinks=true,citecolor=magenta,linkcolor=cyan,backref]{hyperref}
\AtBeginDocument{}

\numberwithin{equation}{section}

\newtheorem{proposition}{Proposition}[section]
\newtheorem{theorem}[proposition]{Theorem}
\newtheorem{lemma}[proposition]{Lemma}

\theoremstyle{definition}
\newtheorem{definition}[proposition]{Definition}
\newtheorem{remark}[proposition]{Remark}
\newtheorem{example}[proposition]{Example}
\newtheorem{question}[proposition]{Question}

\usepackage{color}

\definecolor{gr}{rgb}   {0., 0.8, 0. } 
\definecolor{bl}{rgb}   {0., 0.5, 1. } 
\definecolor{mg}{rgb}   {0.7, 0., 0.7}

\newcommand{\fra}{\mathfrak{a}}

\title{Maximal regularity for non-autonomous\\ Robin boundary conditions}
\author{Wolfgang Arendt\,\thanks{Universit\"at Ulm,
Germany}\,
\and Sylvie Monniaux\,\thanks{Aix-Marseille Universit\'e, CNRS, Centrale Marseille, 
I2M, UMR 7373 -- 13453 Marseille, France}\,  
\footnote{partially supported by the ANR project HAB, ANR-12-BS01-0013-03, the Labex 
Archim\`ede (ANR-11-LABX-0033) and the A*MIDEX project (ANR-11-IDEX-0001-02)}}

\date{}

\begin{document}

\maketitle

\begin{abstract}
We consider a non-autonomous Cauchy problem
$$
\dot u(t)+{\mathcal{A}}(t)u(t)=f(t),\quad u(0)=u_0
$$
where ${\mathcal{A}}(t)$ is associated with the form 
$\fra(t;.,.):V\times V\to{\mathbb{C}}$, where $V$ and $H$ are Hilbert spaces
such that $V$ is continuously and densely embedded in $H$. We prove
$H$-maximal regularity, i.e., the weak solution $u$ is actually in 
$H^1(0,T;H)$ (if $u_0\in V$ and $f\in L^2(0,T;H)$) under a new regularity
condition on the form $a$ with respect to time; namely H\"older continuity
with values in an interpolation space. This result is best suited to treat Robin
boundary conditions. The maximal regularity allows one to use fixed point
arguments to some non linear parabolic problems with Robin boundary conditions.
\end{abstract}


\section{Introduction} 
\label{intro}

In the background of this article is a longstanding problem by J.-L. Lions on
non-autono\-mous forms. We give a solution of the problem in a special case
which is most suitable for treating non-autonomous Robin boundary 
conditions. To be more specific we consider a non-autonomous form
$$
\fra:[0,T]\times V\times V\to{\mathbb{C}}
$$
where $V$ is a Hilbert space continuously and densely embedded into another
Hilbert space $H$. We assume that
\begin{align}
|{\mathfrak{a}}(t;u,v)|\le M\|u\|_V\|v\|_V&\qquad(t\in[0,T],u,v\in V)
\label{bdd}\\[4pt]
\Re e\,{\mathfrak{a}}(t;u,u)\ge \delta \|u\|_V^2&\qquad(t\in [0,T],u\in V)
\label{coer}
\end{align}
for some constants $M,\delta >0$, and that $\fra(.;u,v)$ is measurable for 
all $u,v\in V$. Denote by ${\mathcal{A}}(t):V\to V'$ the operator given by
$$
\langle {\mathcal{A}}(t)u,v\rangle=a(t;u,v),\quad v\in V.
$$
The space
$$
MR(V,V'):=H^1(0,T;V')\cap L^2(0,T;V)
$$
is contained in ${\mathscr{C}}([0,T];H)$ and one has the following 
well-posedness result for weak solutions.

\noindent
{\bf Theorem} (Lions).
{\it For all $f\in L^2(0,T;V')$, $u_0\in H$, there exists a unique 
$u\in MR(V,V')$ solution of}
$$
\dot u(t)+{\mathcal{A}}(t)u(t)=f(t),\quad u(0)=u_0
$$

The letters $MR$ are used to refer to ``maximal regularity"; and indeed
one has maximal regularity in $V'$ in the sense that all three terms 
$\dot u$, ${\mathcal{A}}(\cdot)u(\cdot)$ and $f$ occuring in the equation
belong to $L^2(0,T;V')$. However, considering boundary valued problems
one is interested in {\em strong solutions}, i.e., solutions 
$u\in H^1(0,T;H)$ and not only in $H^1(0,T;V')$ (note that 
$H\hookrightarrow V'$ by the natural embedding).

\medskip

\noindent
{\bf Problem.} Given $f\in L^2(0,T;H)$, $u_0\in H$ good enough, under 
which regularity assumptions on the form $\fra$ is u in
$H^1(0,T;H)$?

\medskip

This problem is explicitely formulated by Lions \cite[p.\,98]{Li61} if 
${\mathfrak{a}}(t;v,w)=\overline{{\mathfrak{a}}(t;w,v)}$ for all $v,w\in V$. 
In general, even for $u_0=0$, \eqref{bdd} and \eqref{coer} are not
sufficient for having $u\in H^1(0,T;V)$. This has been shown recently by 
Dier \cite{Di14}. But several positive answers are given by Lions \cite{Li61}. 
More recently it has been shown that actually $u\in H^1(0,T;H)$ for any 
$u_0\in V$ provided
${\mathfrak{a}}(.;v,w)$ is Lipschitz continuous and symmetric (see \cite{ADLO13}
where also a multiplicative perturbation is admitted) or if ${\mathfrak{a}}(.;v,w)$ is 
symmetric and of bounded variations (see Dier \cite{Di14}). Moreover, for $u_0=0$, 
one has $u\in H^1(0,T;H)$ if ${\mathfrak{a}}(.;v,w)$ is H\"older continuous of
order $\alpha>\frac{1}{2}$ for all $u,v\in V$, see Ouhabaz-Spina \cite{OS10}. 
This has been improved by Haak-Ouhabaz \cite{HO14} where the authors 
remove the symmetry condition and allow non-zero initial conditions. The purpose 
of this article is to establish a different case. We consider $0<\gamma<1$ and the 
complex interpolation space $V_\gamma:=[H,V]_\gamma$. We assume
that ${\mathfrak{a}}$ satisfies \eqref{bdd}, \eqref{coer} and
$$
|{\mathfrak{a}}(t,v,w)-{\mathfrak{a}}(s;v,w)|
\le c |t-s|^\alpha\|v\|_{V}\|w\|_{V_\gamma}
$$
for all $v,w\in V$, $t,s\in [0,T]$, where $\alpha>\frac{\gamma}{2}$. Then we 
show that the solution $u$ from Lions' theorem is actually in $H^1(0,T;H)$
whenever $u_0\in V$. In other words, for all $f\in L^2(0,T,H)$, $u_0\in V$
there is a unique 
$$
u\in MR_{{\mathfrak{a}}}(V,H):=\bigl\{u\in H^1(0,T;H)\cap L^2(0,T;V) :
{\mathcal{A}}(\cdot)u(\cdot)\in L^2(0,T;H)\bigr\}
$$
satisfying
$$
\dot u(t)+{\mathcal{A}}(t)u(t)=f(t)\quad t-a.e.
$$
Thus we have maximal regularity in $H$ in the sense that all three terms
$\dot u$, ${\mathcal{A}}(\cdot)u(\cdot)$ and $f$ are in $L^2(0,T;H)$. 
Similar results can be found in a preprint by Ouhabaz \cite{Ou14} which 
was put on the arXiv website a few weeks after the present work.

Moreover, we show that $MR_{\fra}(V,H)\subset {\mathscr{C}}([0,T];V)$.
Our result can be applied to Robin boundary conditions. If $\Omega$ is 
a bounded Lipschitz domain and 
$$
{\mathcal{B}}:[0,T]\to{\mathscr{L}}(L^2(\partial\Omega))
$$
is H\"older continuous of order $\alpha>\frac{1}{4}$ then given 
$u_0\in H^1(\Omega)$, $f\in L^2(0,T;L^2(\Omega))$ there exists a unique
$u\in H^1(0,T;L^2(\Omega))\cap L^2(0,T;H^1(\Omega))$ such that
$\Delta u\in L^2(0,T;L^2(\Omega))$ and
\begin{align*}
\dot u(t)-\Delta u(t)&=f(t)
\\[4pt]
\partial_\nu u(t)+B(t)u(t)_{|_{\partial\Omega}}&=0
\\[4pt]
u(0)&=u_0.
\end{align*}
This in turn can be used to establish solutions of a non linear problem with
non-autonomous boundary conditions, see Section~\ref{section5}.


\section{Forms, interpolation and square root property}
\label{section1}

Throughout  this paper we consider separable complex Hilbert spaces $V$ and $H$ 
with the property
$V\underset{d}\hookrightarrow H$, i.e., $V$ is densely and continuously 
embedded in $H$. Then, as usual, we have
$$
H\underset{d}\hookrightarrow V'
$$
by associating to $u\in H$ the linear mapping $v\mapsto (v|u)$ where 
$(\cdot|\cdot)$ is the scalar product in $H$ and $V'$ the antidual of $V$.
For $\ell\in[0,1]$, we denote by $V_\ell$ the complex interpolation space 
$[H,V]_\ell$. Thus
$$
V\hookrightarrow V_\ell\hookrightarrow H.
$$
Moreover $V_0=H$, $V_1=V$. Then 
$V_\ell':=(V_\ell)'=[V',H]_{1-\ell}$.
In particular
$$
H\hookrightarrow V_\ell'\hookrightarrow V'
$$
and $V_0'=H$, $V_1'=V'$.

Let ${\mathfrak{a}}:V\times V\to {\mathbb{C}}$ be a sesquilinear form which is 
{\em continuous} (satisfying \eqref{bdd}) and {\em coercive} (satisfying \eqref{coer}).
Then $\langle{\mathcal{A}}u,v\rangle:=\fra(u,v)$ 
defines an invertible operator ${\mathcal{A}}\in{\mathscr{L}}(V,V')$. We denote 
by $A$ {\em the part of ${\mathcal{A}}$ in $H$}, i.e., 
$$
D(A):=\bigl\{u\in V:{\mathcal{A}}u\in H\bigr\},\quad Au:={\mathcal{A}}u.
$$
The operators ${\mathcal{A}}$ and $A$ are sectorial. More precisely there exists 
a sector
$$
\Sigma_\theta:=\bigl\{re^{i\varphi}:r>0,|\varphi|<\theta\bigr\}
$$
with $0\le \theta<\frac{\pi}{2}$ such that $\sigma(A)\subset\Sigma_\theta$, 
$\sigma({\mathcal{A}})\subset\Sigma_\theta$ and for all $\ell\in[0,1]$ and
all $\lambda\notin\Sigma_\theta$,
\begin{align}
\|(\lambda{\rm Id}-{\mathcal{A}})^{-1}\|_{{\mathscr{L}}(V_\ell')}&\le 
\frac{c}{1+|\lambda|},
\label{1.1}\\[4pt]
\|(\lambda{\rm Id}-A)^{-1}\|_{{\mathscr{L}}(V_\ell)}&\le \frac{c}{1+|\lambda|},
\label{1.6}\\[4pt]
\|(\lambda{\rm Id}-{\mathcal{A}})^{-1}\|_{{\mathscr{L}}(V_\ell',V)}&\le 
\frac{c}{(1+|\lambda|)^{\frac{1-\ell}{2}}},
\label{1.2}\\[4pt]
\|(\lambda{\rm Id}-{\mathcal{A}})^{-1}\|_{{\mathscr{L}}(V_\ell',H)}&\le 
\frac{c}{(1+|\lambda|)^{1-\frac{\ell}{2}}},
\label{1.3}\\[4pt]
\|(\lambda{\rm Id}-A)^{-1}\|_{{\mathscr{L}}(H,V_\ell)}&\le 
\frac{c}{(1+|\lambda|)^{1-\frac{\ell}{2}}},
\label{1.5}\\[4pt]
\|(\lambda{\rm Id}-{\mathcal{A}})^{-1}\|_{{\mathscr{L}}(V_\ell',V_\ell)}&\le 
\frac{c}{(1+|\lambda|)^{1-\ell}}.
\label{1.11}
\end{align}
The angle $\theta$ and the constant $c$ merely depend on $\delta$, $M$, $\ell$ 
and the embedding constant $c_H$,
\begin{equation}
\label{VinH}
\|v\|_H\le c_H\|v\|_V, \quad v\in V.
\end{equation}
For the proof of the estimates above, we refer to Tanabe \cite[Chapter~2]{Ta79}, 
or Ouhabaz \cite[Theorem~1.52 and Theorem~1.55]{Ou05} 
(see also Arendt, \cite[Theorem~7.1.4 and Theorem~7.1.5]{Ar05})
We fix an angle $\theta<\vartheta<\frac{\pi}{2}$ and denote by $\Gamma$ 
the contour $\Gamma:=\bigl\{re^{\pm i\vartheta}, r\ge 0\bigr\}$ oriented downwards. 
The operator $-A$ generates a holomorphic $C_0$-semigroup 
$(e^{-tA})_{t\ge0}$ on $H$ given by
\begin{equation}
\label{semigroup}
e^{-tA}=\frac{1}{2\pi i}\int_\Gamma 
e^{-\lambda t}(\lambda{\rm Id}-A)^{-1}\,{\rm d}\lambda,
\end{equation}
and there exists a constant $c>0$ such that 
\begin{equation}
\label{anasg}
\|tAe^{-tA}\|_{{\mathscr{L}}(H)}\le c\quad\mbox{for all }t>0.
\end{equation}
Moreover a theorem by De~Simon \cite[Lemma~3.1]{DeS64} shows
$L^2$ maximal regularity in Hilbert spaces, i.e., for 
holomorphic semigroups on Hilbert spaces there exists a constant $c>0$ such 
that for all $f\in L^2(0,\infty;H)$,
\begin{align}
\label{maxreg-aut}
&t\mapsto\int_0^tAe^{(t-s)A}f(s)\,{\rm d}s\in L^2(0,\infty;H)
\nonumber\\[4pt]
\mbox{and}&\quad\Bigl\|t\mapsto\int_0^tAe^{(t-s)A}f(s)\Bigl\|_{L^2(0,\infty;H)}
\le c\,\|f\|_{L^2(0,\infty;H)}.
\end{align}
Moreover
$$
\|e^{-tA}\|_{{\mathscr{L}}(H)}\le ce^{-\varepsilon t}, \quad t\ge 0,
$$
for some $\varepsilon>0$, $c>0$. Also, the operator $-{\mathcal{A}}$ generates an 
exponentially stable holomorphic $C_0$-semigroup $(e^{-t{\mathcal{A}}})_{t\ge0}$ 
on $V'$. By $\fra^*(u,v):=\overline{\fra(v,u)}$ ($u,v\in V$) we define the form 
$\fra^*$ which is adjoint to $\fra$. Then the operator associated with $\fra^*$ on 
$H$ coincides with the adjoint $A^*$ of $A$.
We define the operator $A^{-\frac{1}{2}}\in{\mathscr{L}}(H)$ by
\begin{equation}
\label{sqrtA}
A^{-\frac{1}{2}}u=\frac{1}{\sqrt{\pi}}\int_0^\infty t^{-\frac{1}{2}}e^{-tA}u\,{\rm d}t,
\end{equation}
and we let $D(A^{\frac{1}{2}}):=A^{-\frac{1}{2}}H$. One has 
$D((\mu{\rm Id}+A)^{\frac{1}{2}})=D(A^{\frac{1}{2}})$ for all $\mu\ge 0$ (see 
\cite[Proposition~3.8.2, p.\,165]{ABHN11}). The domain of $A^{\frac{1}{2}}$ 
is of importance since it describes the initial values $u_0$ for which the Cauchy 
problem
$$
\dot u +Au=0,\quad u(0)=u_0
$$
has an $H^1$-solution. In fact, $u(t):=e^{-tA}u_0$ is the mild solution of this 
problem which is defined for all $u_0\in H$. Lions and Magenes 
(see \cite[Th\'eor\`eme~10.1]{LM68}) showed the equivalence
\begin{equation}
\label{maxreg-aut-u0}
u\in H^1(0,T;H) \quad \mbox{if and only if}\quad u_0\in D(A^{\frac{1}{2}}).
\end{equation}

The space $V$ is in general known, it is typically a Sobolev space as 
$H^1(\Omega)$ or $H^1_0(\Omega)$. However, the right space 
$D(A^{\frac{1}{2}})$ for the admissible initial values is in general different from 
$V$. We introduce a name to describe the important property that both spaces 
coincide.

\begin{definition}
\label{sqrt-property}
The form $\fra$ has the {\em square root property} if $D(A^{\frac{1}{2}})=V$.
\end{definition}

We give an abstract criterion for a particular case where the square root 
property holds.

\begin{example}
Assume that $\fra$ can be written in the form $\fra=\fra_1+\fra_2$ where 
$\fra_1:V\times V\to {\mathbb{C}}$ is bounded and symmetric and 
$\fra_2:V\times H\to{\mathbb{C}}$ is bounded. Then $\fra$ has the square root property.
See McIntosh \cite{McI72}.
\end{example}

Not each form has the square root property. The famous solution of the Kato 
square root problem says that elliptic forms describing a second order differential 
operator with measurable coefficients on bounded open sets of ${\mathbb{R}}^N$
with Dirichlet of Neumann boundary conditions have the square root property 
(see \cite{AHLMcIT02} for the case of $\Omega={\mathbb{R}}^N$ and 
\cite{AT03} for the case of strongly Lipschitz domains). 
We will need the following result by J.-L.~Lions \cite[Th\'eor\`eme~5.1]{Li62}.

\begin{lemma}
\label{sqrt-lions}
The form $\fra$ has the square root property if and only if 
$D(A^{\frac{1}{2}})\subset V$ and $D(A^{*{\frac{1}{2}}})\subset V$.
\end{lemma}

Note that $e^{-t{\mathcal{A}}}_{|_H}=e^{-tA}$ for all $t\ge0$. We obtain from 
\eqref{1.1}--\eqref{1.11} and \eqref{anasg} the following estimates for the semigroup.
\begin{proposition}
\label{pense-bete}
There exists $c>0$ such that for all $\ell\in[0,1]$ and all $t>0$,
\begin{align}
\|e^{-t{\mathcal{A}}}\|_{{\mathscr{L}}(V_\ell',V)} &\le 
\frac{c}{t^{\frac{1+\ell}{2}}}
\label{1.12}\\[4pt]
\|e^{-tA}\|_{{\mathscr{L}}(V_\ell,V)} &\le \frac{c}{t^{\frac{1-\ell}{2}}}
\label{1.13}\\[4pt]
\|e^{-t{\mathcal{A}}}\|_{{\mathscr{L}}(V_\ell',H)} &\le 
\frac{c}{t^{\frac{\ell}{2}}}
\label{1.14}
\end{align}
\end{proposition}

We will consider form perturbations which are continuous on 
$V\times V_\gamma$ and on $V_\gamma\times V$. 
They preserve the square root property.

\begin{proposition}
\label{sqrt-perturbed}
Let ${\mathfrak{a}}_1,{\mathfrak{a}}_2:V\times V\to {\mathbb{C}}$ be two bounded, 
coercive forms. Assume that there exists a constant $c>0$ such that
$$
|{\mathfrak{a}}_1(u,v)-{\mathfrak{a}}_2(u,v)|\le 
c\bigl\|u\|_{V}\|v\|_{V_\gamma}
\quad u,v\in V,
$$
where $0\le \gamma<1$. Then, ${\mathfrak{a}}_1$ has the square root property if, 
and only if, ${\mathfrak{a}}_2$ has it.
\end{proposition}

In the following proof the constant $c>0$ will vary from one line to the other but 
does not depend on the variables to be estimated. We keep this convention 
throughout the paper.

\begin{proof}
By hypothesis we have $A_1^{-1/2}H\subset V$. We show that 
$(A_1^{-1/2}-A_2^{-1/2})H\subset V$. Let $u\in H$. Then
\begin{align*}
A_1^{-1/2}u-A_2^{-1/2}u&=
\frac{1}{\sqrt{\pi}}\int_0^\infty\frac{1}{\sqrt{\sigma}}(e^{-\sigma A_1}u-
e^{-\sigma A_2}u)\,{\rm d}\sigma
\\[4pt]
&=\frac{1}{\sqrt{\pi}}\int_0^\infty\frac{1}{\sqrt{\sigma}}\,\frac{1}{2\pi i}
\int_\Gamma e^{-\sigma\lambda}
\bigl((\lambda{\rm Id}-A_1)^{-1}u-(\lambda{\rm Id}-A_2)^{-1}u
\bigr)\,{\rm d}\lambda\,{\rm d}\sigma
\\[4pt]
&=\frac{1}{\sqrt{\pi}}\int_0^\infty\frac{1}{\sqrt{\sigma}}\,\frac{1}{2\pi i}
\int_\Gamma e^{-\sigma\lambda}
(\lambda{\rm Id}-{\mathcal{A}}_1)^{-1}({\mathcal{A}}_1-{\mathcal{A}}_2)
(\lambda{\rm Id}-A_2)^{-1}u\,{\rm d}\lambda\,{\rm d}\sigma.
\end{align*}
Since ${\mathcal{A}}_1-{\mathcal{A}}_2\in{\mathscr{L}}(V,V_\gamma')$ 
and \eqref{1.5} for $\ell=1$ and \eqref{1.2} for $\ell=\gamma$
$$
\begin{array}{lrcll}
&\|(\lambda{\rm Id}-A_2)^{-1}\|_{{\mathscr{L}}(H,V)}&\le& 
\frac{c}{(1+|\lambda|)^{1/2}}&
\\[4pt]
\mbox{and }&&&&\quad \mbox{ for all }\lambda\in\Gamma
\\[4pt]
&\|(\lambda{\rm Id}-{\mathcal{A}}_1)^{-1}\|_{{\mathscr{L}}(V_\gamma',V)}
&\le& \frac{c}{(1+|\lambda|)^{\frac{1-\gamma}{2}}}& 
\end{array}
$$
we see that
$(\lambda{\rm Id}-{\mathcal{A}}_1)^{-1}({\mathcal{A}}_1-{\mathcal{A}}_2)
(\lambda{\rm Id}-A_2)^{-1}u \in V$ and the integral converges in $V$. In fact,
\begin{equation}
\label{sqrtA1-sqrtA2}
\bigl\|(\lambda{\rm Id}-{\mathcal{A}}_1)^{-1}({\mathcal{A}}_1-{\mathcal{A}}_2)
(\lambda{\rm Id}-A_2)^{-1}u\bigr\|_V
\le \frac{c}{(1+|\lambda|)^{1-\gamma/2}}\,\|u\|_H
\end{equation}
and therefore
\begin{align*}
\|A_1^{-1/2}u-A_2^{-1/2}u\|_V&\le 
c\int_0^\infty\frac{1}{\sqrt{\sigma}}\Bigl|\int_\Gamma e^{-\sigma\Re e\,\lambda}
\,\frac{1}{(1+|\lambda|)^{\frac{1-\gamma}{2}}}\,
\frac{1}{(1+|\lambda|)^{1/2}}\|u\|_H
\,{\rm d}|\lambda|\Bigr|\,{\rm d}\sigma
\\[4pt]
&\le
c \Bigl(\int_0^\infty\frac{1}{\sqrt{\sigma}}
\int_0^\infty e^{-\sigma r\cos \vartheta}\,\frac{1}{(1+r)^{1-\gamma/2}}
\,{\rm d}r\,{\rm d}\sigma\Bigr)\|u\|_H
\\[4pt]
&\le c\|u\|_H
\int_0^\infty\Bigl(\int_0^\infty\frac{1}{\sqrt{\sigma}}\,
e^{-\sigma r\cos \vartheta}\,{\rm d}\sigma\Bigr)
\,\frac{1}{(1+r)^{1-\gamma/2}}\,{\rm d}r
\\[4pt]
&\le c\|u\|_H
\int_0^\infty\Bigl(\int_0^\infty\frac{\sqrt{r}}{\sqrt{s}}\,e^{-s\cos \vartheta}\,
\frac{1}{r}\,{\rm d}s\Bigr)
\,\frac{1}{(1+r)^{1-\gamma/2}}\,{\rm d}r
\\[4pt]
&\le c\|u\|_H \int_0^\infty \frac{1}{\sqrt{r}}\,\frac{1}{(1+r)^{1-\gamma/2}}
\,{\rm d}r \ \le\ c\|u\|_H.
\end{align*}
Thus the claim is proved and $D(A_2^{\frac{1}{2}})\subset V$. Applying this result 
to $({\mathfrak{a}}_1^*,{\mathfrak{a}}_2^*)$ instead of 
$({\mathfrak{a}}_1,{\mathfrak{a}}_2)$ we find that 
$D(A_2^{*\frac{1}{2}})\subset V$. Indeed we have 
${\mathcal{A}}_1^*-{\mathcal{A}}_2^*\in{\mathscr{L}}(V_\gamma,V')$. 
Thanks to \eqref{1.2} for $\ell=1$ and \eqref{1.5} for $\ell=\gamma$, 
\eqref{sqrtA1-sqrtA2} becomes
$$
\bigl\|(\lambda{\rm Id}-{\mathcal{A}}_1^*)^{-1}({\mathcal{A}}_1^*-{\mathcal{A}}_2^*)
(\lambda{\rm Id}-A_2^*)^{-1}u\bigr\|_V
\le \frac{c}{(1+|\lambda|)^{1-\gamma/2}}\,\|u\|_H
$$
It follows from Lemma~\ref{sqrt-lions} that ${\mathfrak{a}}_2$ has the 
square root property.
\end{proof}


\section{Non-autonomous forms}
\label{section2}

In this section, we consider a time-dependent form $\fra$. Let $V$, $H$ be 
separable complex Hilbert spaces. Let $T>0$ and let
\begin{align}
&{\mathfrak{a}}(t;\cdot,\cdot):\times V\times V\to {\mathbb{C}}\quad
\mbox{be a sesquilinear form for all $t\in[0,T]$ satisfying}
\nonumber\\[4pt]
&\eqref{bdd} \mbox{ ({\em boundedness}) } \mbox{ and }\eqref{coer}
\mbox{ ({\em coercivity})}
\label{bddcoer}\\[4pt]
&{\mathfrak{a}}(\cdot,;u,v):[0,T]\to{\mathbb{C}}\mbox{ is measurable for all }
u,v\in V. 
\label{measurable}
\end{align}
Then for each $t\in [0,T]$ we consider the operator 
${\mathcal{A}}(t)$ on $V'$ which is associated with $\fra(t;\cdot,\cdot)$ and we 
denote by $A(t)$ the part of ${\mathcal{A}}(t)$ in $H$. 
A classical theorem of Lions (see \cite[Th\'eor\`eme~1 p.\,619, Th\'eor\`eme~2 
p.\,620, Chap.\.XVIII~\S3]{DL88}, \cite[Proposition~2.3, Chap. III.2]{Sho97}) 
establishes well-posedness and maximal regularity in $V'$ of the Cauchy problem
\begin{equation}
\label{CP}
\left\{\begin{array}{rcl}
\dot u(t)+{\mathcal{A}}(t)u(t)&=&f(t)
\\[4pt]
u(0)&=&u_0.
\end{array}\right.
\end{equation}
More precisely, we let
$$
MR(V,V'):=H^1(0,T;V')\cap L^2(0,T;V).
$$
Then $MR(V,V')\subset{\mathscr{C}}([0,T];H)$ and the following holds.

\begin{theorem}[Lions]
\label{lions}
Let $f\in L^2(0,T;V')$, $u_0\in H$. Then there exists a unique solution 
$u\in MR(V,V')$ of \eqref{CP}.
\end{theorem}

The operator ${\mathcal{A}}(t)$ is not the real object of interest if one considers 
boundary value problems (see Section~\ref{section5}), it is rather its part in $H$ 
which realizes the boundary conditions. So the following question is of great 
importance.

\begin{question}
Assume that $f\in L^2(0,T;H)$ and $u_0\in V$. Does it follow that the solution 
$u\in MR(V,V')$ of \eqref{CP} is actually in $H^1(0,T;H)$?
\end{question}

In that case, since $u$ is a solution, $u(t)\in D(A(t))$ $a.e.$ and 
$\dot u(t)+A(t)u(t)=f(t)$.
We have seen that we have to impose at least that $\fra(0;\cdot,\cdot)$ has the
square root property (since otherwise, even for $f\equiv 0$ and even for 
$A(t)\equiv A(0)$ there exists a counterexample). Our aim is to give a positive 
answer to the question if $\fra$ satisfies some further regularity in time.

\begin{definition}
The form $\fra$ (or Problem \eqref{CP}) satisfies {\em maximal regularity in $H$} 
if for each $f\in L^2(0,T;H)$ and for each $u_0\in V$, the solution $u\in MR(V,V')$ 
of \eqref{CP} is actually in $H^1(0,T;H)$.
\end{definition}

The problem \eqref{CP} is invariant under shifting the operator by a scalar 
operator as the following proposition shows.

\begin{proposition}
\label{shift-inv}
Let $\mu\in {\mathbb{R}}$.
\begin{enumerate}
\item
For each $f\in L^2(0,T;V')$ and $u_0\in H$ there is a unique $u\in MR(V,V')$ such 
that
\begin{equation}
\label{CPshifted}
\left\{\begin{array}{rcl}
\dot u(t)+{\mathcal{A}}(t)u(t)+\mu u(t)&=&f(t)\quad a.e.,
\\[4pt]
u(0)&=&u_0.
\end{array}\right.
\end{equation}
\item
If problem \eqref{CP} has maximal regularity in $H$ and if $u_0\in V$ and 
$f\in L^2(0,T;H)$, then this solution $u$ of \eqref{CPshifted} belongs to 
$H^1(0,T;H)$.
\end{enumerate}
\end{proposition}

\begin{proof}
\begin{enumerate}
\item
Let $v$ be the unique solution in $MR(V,V')$ of
$$
\left\{\begin{array}{rcl}
\dot v(t)+{\mathcal{A}}(t)v(t)&=&e^{\mu t}f(t)\quad a.e.
\\[4pt]
v(0)&=&u_0
\end{array}\right.
$$
and let $u$ be defined by $u(t)=e^{-\mu t}v(t)$, $t\in [0,T]$. It is immediate that 
$u\in MR(V,V')$,
$u(0)=u_0$ and
\begin{align*}
\dot u(t)&=-\mu e^{-\mu t}v(t)+e^{-\mu t}\dot v(t)
\\[4pt]
&=-\mu u(t)+e^{-\mu t}\bigl(-{\mathcal{A}}(t)v(t)+f(t)e^{\mu t}\bigr)
\\[4pt]
&=-\mu u(t)-{\mathcal{A}}(t)u(t)+f(t),\quad a.e.\ t\in[0,T]
\end{align*}
which proves that $u$ satisfies \eqref{CPshifted}. Assume now that 
\eqref{CPshifted} admits two solutions in $MR(V,V')$ $u_1$ and $u_2$. Then
$v_1(t)=e^{\mu t}u_1(t)$ and $v_2(t)=e^{\mu t}v_2(t)$ define two solutions of 
\eqref{CP} in $MR(V,V')$ with initial value $u_0$ and $e^{\mu\cdot}f(\cdot)$ instead
of $f$. Therefore they coincide by Lions' Theorem~\ref{lions}.
\item
Assume now that $u$ is a solution in $MR(V,V')$ of \eqref{CPshifted}. 
Let $v:t\mapsto e^{\mu t}u(t)$; $v$ is the unique solution of \eqref{CP} in 
$MR(V,V')$ with $g=e^{\mu\cdot}f(\cdot)$ instead of $f$.
Since problem~\eqref{CP} has the maximal regularity property and 
$g\in L^2(0,T;H)$, $v(0)=u_0\in V$, the solution $v$ belongs to $H^1(0,T;H)$ 
and $t\mapsto {\mathcal{A}}(t)v(t)\in L^2(0,T;H)$.
This proves that $u:t\mapsto e^{-\mu t}v(t)$ also belongs to $H^1(0,T;H)$ (since
$\dot u(t)=\mu e^{-\mu t}v(t)+e^{-\mu t}\dot v(t)$) and 
$t\mapsto {\mathcal{A}}(t)u(t)=e^{-\mu t}{\mathcal{A}}(t)v(t) \in L^2(0,T;H)$.
\end{enumerate}
\end{proof}

\noindent
Finally, we establish a representation formula of the solution 
$u\in MR(V,V')$ of \eqref{CP}.

\begin{proposition}
\label{TLJ}
\begin{enumerate}
\item
Let $f\in L^2(0,T;V')$, $u_0\in H$. Let $u\in MR(V,V')$ be the solution 
of \eqref{CP}. Then
\begin{align}
\label{FTLJ}
&u(t)=e^{-tA(t)}u_0+\int_{0}^t e^{-(t-s){\mathcal{A}}(t)}f(s)\,{\rm d}s+
\int_{0}^t e^{-(t-s){\mathcal{A}}(t)}\bigl({\mathcal{A}}(t)-{\mathcal{A}}(s)\bigr)
u(s)\,{\rm d}s,
\nonumber\\[4pt]
&\mbox{for all $t\in[0,T]$.}
\end{align}
\item
Moreover, there is only one $u\in MR(V,V')$ satisfying this identity if we assume 
in addition that $t\mapsto {\mathcal{A}}(t)\in{\mathscr{L}}(V,V')$ is Dini-continuous, 
i.e., admits a modulus of continuity $\omega$ in the operator norm with the property 
that $t\mapsto \frac{1}{t}\,\omega(t)\in L^1(0,T)$.
\end{enumerate}
\end{proposition}

\begin{proof}
\begin{enumerate}
\item
This formula already appears in \cite[formula~(1.18), p.\,57]{AT87} for operators 
with different properties; see also \cite[Lemma~2.4]{HO14}).
Let $0<t\le T$. Consider the function 
$v:[0,t]\ni s\mapsto e^{-(t-s)A(t)}u(s)$. Then 
$v\in{\mathscr{C}}([0,t];H)\cap H^1(0,t;V')$ and
\begin{align}
\dot v(s)&=A(t)e^{-(t-s)A(t)}u(s)+e^{-(t-s){\mathcal{A}}(t)}\dot u(s)
\nonumber\\[4pt]
&=e^{-(t-s){\mathcal{A}}(t)}\Bigl({\mathcal{A}}(t)u(s)+\bigl(-{\mathcal{A}}(s)u(s)
+f(s)\bigr)\Bigr)
\nonumber\\[4pt]
&=e^{-(t-s){\mathcal{A}}(t)}\bigl({\mathcal{A}}(t)-{\mathcal{A}}(s)\bigr)u(s)
+e^{-(t-s){\mathcal{A}}(t)}f(s).
\label{v'}
\end{align}
Thus integrating between $0$ and $t$ gives 
$\displaystyle{v(t)=\int_{0}^t\dot v(s)\,{\rm d}s +v(0)}$ which is the claim.
\item
In order to prove uniqueness, assume that there are two functions $u_1$ and $u_2$ 
in the space $MR(V,V')$ satisfying \eqref{FTLJ} and denote by $w$ the difference 
$u_1-u_2$. Then $w\in MR(V,V')$ and satisfies
\begin{equation}
\label{FTLJbis}
w(t)=\int_{0}^t e^{-(t-s){\mathcal{A}}(t)}\bigl({\mathcal{A}}(t)-{\mathcal{A}}(s)
\bigr)w(s)\,{\rm d}s,
\end{equation}
for all $t\in[0,T]$. Let $\tau:=\sup\bigl\{t\in[0,T];w(s)=0\mbox{ on }[0,t]\bigr\}$.
Assume that $\tau<T$. Using \eqref{FTLJbis}, 
the continuity properties of ${\mathcal{A}}(\cdot)$ and the estimate \eqref{1.12} for
$\ell=1$, we obtain for all $\tau\le t\le T$
$$
\|w(t)\|_V\le \int_{\tau}^t\frac{c}{t-s}\,\omega(t-s)\|w(s)\|_V\,{\rm d}s
$$
and therefore by Young's inequality for convolution
$$
\|w\|_{L^2(\tau,t;V)}\le c\,\Bigl(\int_0^{t-\tau}\frac{\omega(s)}{s}\,{\rm d}s\Bigr) 
\|w\|_{L^2(\tau,t;V)}.
$$
Choosing $\varepsilon$ small enough so that $\tau+\varepsilon\le T$ and
$c\,\bigl(\int_0^{\varepsilon}\frac{\omega(s)}{s}\,{\rm d}s\bigr)<1$, we proved that 
$w(t)=0$ almost everywhere on $[\tau,\tau+\varepsilon]$. Since
$MR(V,V')\subset{\mathscr{C}}([0,T];H)$, we have then $w(t)=0$
everywhere on $[\tau,\tau+\varepsilon]$, which contradict the
definition of $\tau$. This proves that $\tau=T$ and ultimately $w(t)=0$ on $[0,T]$.
\end{enumerate}
\end{proof}


\section{Maximal regularity in $H$}
\label{section3}

Let $V$ and $H$ be two separable complex Hilbert spaces such that
$V\underset{d}\hookrightarrow H$. Let $\fra:[0,T]\times V\times V\to{\mathbb{C}}$
be a non-autonomous form satisfying \eqref{bddcoer} and \eqref{measurable}. 
We denote by ${\mathcal{A}}(t)$ the operator on $V'$ associated with 
$\fra(t;.,.)$ and by $A(t)$ its part in $H$. The essential further condition 
concerns continuity in time. We assume that there exist $0\le\gamma<1$ and a 
modulus of continuity $\omega$ such that
\begin{equation}
\label{omega}
|\fra(t;u,v)-\fra(s;u,v)|\le \omega(|t-s|)\|u\|_{V}\|v\|_{V_\gamma}
\end{equation}
for all $t,s\in[0,T]$, $u,v\in V_\gamma$. We suppose that 
$\omega:[0,T]\to[0,+\infty)$
is continuous and satisfies
\begin{align}
&\sup_{t\in[0,T]}\frac{\omega(t)}{t^{\gamma/2}}<\infty
\label{omega-bdd}\\[4pt]
\mbox{and}&\quad\int_0^T\frac{\omega(t)}{t^{1+\gamma/2}}\,{\rm d}t<\infty.
\label{omegaL1}
\end{align}
The main example of such a continuity modulus is the function $\omega(t)=t^\alpha$ 
with $\alpha>\frac{\gamma}{2}$. We remark that conditions \eqref{omega-bdd},
\eqref{omegaL1} imply that
\begin{equation}
\label{omegaL2}
\int_0^T\frac{\omega(t)^2}{t^{1+\gamma}}\,{\rm d}t<\infty.
\end{equation}
Finally, we impose that $\fra(0;.,.)$ has the square root property. By 
Proposition~\ref{sqrt-perturbed} this implies that $\fra(t;.,.)$ has the 
square root property for all $t\in [0,T]$. Under the preceding conditions we 
have the following result on {\em maximal regularity in $H$}.

\begin{theorem}
\label{maxreg-nonaut}
Assume that $\fra(0;.,.)$ has the square root property.
Let $u_0\in V$, $f\in L^2(0,T;H)$. Then there exists a unique 
$u\in H^1(0,T;H)\cap L^2(0,T;V)$ such that $u(t)\in D(A(t))$ $a.e.$ and
\begin{equation}
\label{Cauchy}
\left\{\begin{array}{rcl}
\dot u(t)+A(t)u(t)&=&f(t)\quad a.e.
\\[4pt]
u(0)=u_0.
\end{array}\right.
\end{equation}
\end{theorem}

Thus the solution $u$ is in the space
$$
MR_\fra:=\bigl\{u\in H^1(0,T;H)\cap L^2(0,T;V): u(0)\in V, 
{\mathcal{A}}(\cdot)u(\cdot)\in L^2(0,T;H)\bigr\}.
$$
We will see below that $MR_\fra\subset{\mathscr{C}}([0,T];V)$.

\begin{remark}
\begin{itemize}
\item[(a)]
The space $MR_\fra$ endowed with the norm 
$$
\|u\|_{MR_\fra}=\|\dot u\|_{L^2(0,T;H)}+\|A(\cdot)u(\cdot)\|_{L^2(0,T;H)}+\|u(0)\|_V
$$
is a Banach space.
\item[(b)]
It follows from the Closed Graph Theorem that there exists a constant $c>0$ 
such that 
$$
\|u\|_{H^1(0,T;H)}+\|{\mathcal{A}}(\cdot)u(\cdot)\|_{L^2(0,T;H)}
\le c\,\bigl(\|u_0\|_V+\|f\|_{L^2(0,T;H)}\bigr)
$$
for each $u_0\in V$ and $f\in L^2(0,T;H)$, where $u$ is the solution of 
\eqref{Cauchy}.
\end{itemize}
\end{remark}

\begin{proof}[Proof of Theorem~\ref{maxreg-nonaut}]
By Lions' Theorem~\ref{lions} there exists a unique solution $u\in MR(V,V')$ of 
the problem. We have to show that ${\mathcal{A}}(\cdot)u(\cdot)\in L^2(0,T;H)$. 
For that we use the decomposition \eqref{FTLJ} and show that 
${\mathcal{A}}(\cdot)u_j(\cdot)\in L^2(0,T;H)$ for $j=1,2,3$
where
\begin{align*}
&u_1(t)=e^{-tA(t)}u_0,\quad u_2(t)=\int_0^te^{-(t-s)A(t)}f(s)\,{\rm d}s,
\\[4pt]
&\mbox{and}
\quad u_3(t)=\int_0^te^{-(t-s){\mathcal{A}}(t)}\bigl({\mathcal{A}}(t)-{\mathcal{A}}(s)\bigr)u(s)\,{\rm d}s.
\end{align*}
Remark that \eqref{omega} implies for all $t,s\in[0,T]$
\begin{equation}
\label{A(t)-A(s)}
{\mathcal{A}}(t)-{\mathcal{A}}(s)\in{\mathscr{L}}(V,V_\gamma')\quad\mbox{and}
\quad \|{\mathcal{A}}(t)-{\mathcal{A}}(s)\|_{{\mathscr{L}}(V,V_\gamma')}
\le c\,\omega(|t-s|).
\end{equation}
We divide the proof into three steps.

\noindent
{\tt Step 1:}
We adapt the proof of \cite[Lemma~2.7]{HO14} to our situation.
Since ${\mathfrak{a}}(0;.,.)$ has the square root property, thanks to 
\eqref{maxreg-aut-u0},
$t\mapsto A(0)e^{-tA(0)}u_0\in L^2(0,T;H)$. Thus it suffices to show that
$$
\phi:t\mapsto A(t)e^{-tA(t)}u_0-A(0)e^{-tA(0)}u_0\in L^2(0,T;H).
$$
Using
\begin{equation}
\label{3.6}
(\lambda{\rm Id}-A(t))^{-1}-(\lambda{\rm Id}-A(0))^{-1}
=(\lambda{\rm Id}-{\mathcal{A}}(t))^{-1}({\mathcal{A}}(t)-{\mathcal{A}}(0))
(\lambda{\rm Id}-A(0))^{-1}
\end{equation}
we see that
$$
\phi(t)=\frac{1}{2\pi i}\int_\Gamma \lambda e^{-\lambda t}
(\lambda{\rm Id}-A(t))^{-1}({\mathcal{A}}(t)-{\mathcal{A}}(0))
(\lambda{\rm Id}-A(0))^{-1}u_0\,{\rm d}\lambda.
$$
Using \eqref{1.5} for $\ell=\gamma$ and \eqref{1.6} for $\ell=1$ we estimate
\begin{align*}
\|\phi(t)\|_H&\le
c\,\int_\Gamma|\lambda|e^{-t\Re e\,\lambda}\,
\frac{1}{|\lambda|^{1-\gamma/2}}\,\omega(t)
\|(\lambda{\rm Id}-A(0))^{-1}u_0\|_{V}\,{\rm d}|\lambda|
\\[4pt]
&\le c\,\omega(t)\int_\Gamma|\lambda|e^{-t\Re e\,\lambda}\,
\frac{1}{|\lambda|^{1-\gamma/2}}\,
\frac{1}{|\lambda|}\|u_0\|_V\,{\rm d}|\lambda|
\\[4pt]
&\le c\,\omega(t)\|u_0\|_V\int_0^\infty e^{-tr\cos\vartheta}\,
\frac{1}{r^{1-\gamma/2}}\,{\rm d}r
\\[4pt]
&\le c\,\omega(t)\|u_0\|_V\int_0^\infty e^{-\rho\cos\vartheta}\,
\frac{t^{1-\gamma/2}}{\rho^{1-\gamma/2}}\,\frac{1}{t}\,{\rm d}\rho
\ \le\ c\,\frac{\omega(t)}{t^{\gamma/2}}\,\|u_0\|_V.
\end{align*}
It follows from \eqref{omega-bdd} that $\phi\in L^2(0,T;H)$.

\medskip

\noindent
{\tt Step 2:}
We show that $t\mapsto\int_0^t A(t)e^{-(t-s)A(t)}f(s)\,ds\in L^2(0,T;H)$.
The proof here is much more elementary than the one of 
\cite[Lemma~2.5]{HO14} thanks to our stronger condition \eqref{omega} on 
the form ${\mathfrak{a}}$.
By \eqref{maxreg-aut} $A(0)$ satisfies maximal regularity and
$$
t\mapsto \int_0^t A(0)e^{-(t-s)A(0)}f(s)\,{\rm d}s \in L^2(0,T;H).
$$
Thus it suffices to show that
$$
\phi:t\mapsto \int_0^t A(t)e^{-(t-s)A(t)}f(s)\,{\rm d}s
-\int_0^t A(0)e^{-(t-s)A(0)}f(s)\,{\rm d}s \in L^2(0,T;H).
$$
As before we have
$$
\phi(t)=\int_0^t\frac{1}{2\pi i}\int_\Gamma \lambda e^{-(t-s)\lambda}
(\lambda{\rm Id}-{\mathcal{A}}(t))^{-1}({\mathcal{A}}(t)-{\mathcal{A}}(0))
(\lambda{\rm Id}-A(0))^{-1}f(s)\,{\rm d}\lambda\,{\rm d}s.
$$
Using \eqref{1.3} for $\ell=\gamma$ and \eqref{1.5} for $\ell=1$ we obtain
\begin{align*}
\|\phi(t)\|_H&\le c\int_0^t\int_0^\infty re^{-(t-s)r\cos\vartheta}\,
\frac{1}{r^{1-\gamma/2}}\,\omega(t)\,
\frac{1}{r^{1/2}}\,\|f(s)\|_H\,{\rm d}r\,{\rm d}s
\\[4pt]
&= c\,\omega(t)\int_0^t\Bigl(\int_0^\infty e^{-\rho\cos\vartheta}\,
\rho^{\frac{\gamma-1}{2}}\,
\,{\rm d}\rho\Bigr)(t-s)^{-\frac{1+\gamma}{2}}\|f(s)\|_H\,{\rm d}s
\\[4pt]
&= c\,\omega(t)\bigl(h*\|f(\cdot)\|_H\bigr)(t)\quad \mbox{for all }0\le t\le T
\end{align*}
where $h$ is defined by
$h(t)=\left\{\begin{array}{ll}
0&\mbox{for }t\le0
\\[4pt]
t^{-\frac{1+\gamma}{2}}&\mbox{for }0<t\le T
\\[4pt]
0&\mbox{for }t\ge T
\end{array}\right.$
is in $L^1({\mathbb{R}})$ since $\frac{1+\gamma}{2}<1$. It follows that 
$$
\int_0^T\|\phi(t)\|^2_H\,{\rm d}t<\infty.
$$
{\tt Step 3:}
In order to show that $A(\cdot)u_3(\cdot)\in L^2(0,T;H)$, we define for 
$g\in L^2(0,T;H)$,
\begin{align*}
(Qg)(t)&:=\int_0^t A(t)e^{-(t-s)A(t)}\bigl({\mathcal{A}}(t)-{\mathcal{A}}(s)\bigr)
A(s)^{-1}g(s)\,{\rm d}s
\\[4pt]
&=\int_0^t A(t)e^{-\frac{t-s}{2}A(t)}e^{-\frac{t-s}{2}{\mathcal{A}}(t)}
\bigl({\mathcal{A}}(t)-{\mathcal{A}}(s)\bigr)A(s)^{-1}g(s)\,{\rm d}s.
\end{align*}
The proof here is very much inspired by \cite[end of \S1]{AT87} (see
also \cite[beginning of \S3]{HM00}, \cite[Proof of Theorem~1.2]{HO14}).
Let $\varepsilon>0$. Replacing $A(s)$ by $A(s)+\mu{\rm Id}$ 
(see Proposition~\ref{shift-inv}) we may assume that 
$\|A(s)^{-1}\|_{{\mathscr{L}}(H,V)}\le \varepsilon$ (see \eqref{1.5} for $\ell=1$)
for all $s\ge0$. Thus by \eqref{1.14} for $\ell=\gamma$ and \eqref{anasg} we 
have the following estimates
\begin{align*}
\|Qg(t)\|_H&\le
\int_0^t \|A(t)e^{-\frac{t-s}{2}A(t)}\|_{{\mathscr{L}}(H)}
\|e^{-\frac{t-s}{2}{\mathcal{A}}(t)}\|_{{\mathscr{L}}(V_\gamma',H)}
\,\omega(t-s)\, \varepsilon\,\|g(s)\|_H\,{\rm d}s
\\[4pt]
&\le c\,\varepsilon \int_0^t\frac{1}{t-s}\,\frac{1}{(t-s)^{\gamma/2}}\,
\omega(t-s)\|g(s)\|_H\,{\rm d}s.
\end{align*}
Since $k(t)=\frac{1}{t^{1+\gamma/2}}\,\omega(t)$ defines a function 
$k\in L^1(0,T)$ thanks to \eqref{omegaL1}, it follows that $Qg\in L^2(0,T;H)$ and
$$
\|Qg\|_{L^2(0,T;H)}\le c\,\varepsilon\,\|g\|_{L^2(0,T;H)}.
$$
Choosing $\varepsilon>0$ small enough we can arrange that 
$\|Q\|_{{\mathscr{L}}(L^2(0,T;H))}\le \frac{1}{2}$.
Thus ${\rm Id}-Q$ is invertible. By Step~1 and Step~2 we know that 
$$
h:=A(\cdot)\bigl(u_1(\cdot)+u_2(\cdot)\bigr)\in L^2(0,T;H).
$$
Let $w=A(\cdot)^{-1}({\rm Id}-Q)^{-1}h$. Then $A(\cdot)w(\cdot)\in L^2(0,T;H)$. 
Since $h=A(\cdot)w(\cdot)-Q\bigl(A(\cdot)w(\cdot)\bigr)$ one has
\begin{align*}
&A(t)e^{-tA(t)}u_0+A(t)\int_0^t e^{-(t-s)A(t)}f(s)\,{\rm d}s
+A(t)\int_0^t e^{-(t-s)A(t)}\bigl({\mathcal{A}}(t)-{\mathcal{A}}(s)\bigr)w(s)\,{\rm d}s 
\\[4pt]
&=\ A(t)w(t).
\end{align*}
Applying $A(t)^{-1}$ on both sides we see from Proposition~\ref{TLJ} that $w=u$. 
Hence $A(\cdot)u(\cdot)=A(\cdot)w(\cdot)\in L^2(0,T;H)$.
\end{proof}

\begin{remark}
If we do not suppose the square root property, then the proof of
Theorem~\ref{maxreg-nonaut} shows that for $u_0\in D(A(0)^{\frac{1}{2}})$,
$f\in L^2(0,T;H)$, the solution $u$ given by Lions' Theorem is in
$H^1(0,T;H)$, i.e., we have the same conclusion as in 
Theorem~\ref{maxreg-nonaut} if we choose the right trace space.
\end{remark}

As announced above, we now show that

\begin{theorem}
\label{MRcont}
The space $MR_{\fra}$ is continuously embedded into ${\mathscr{C}}([0,T],V)$.
\end{theorem}

To prove this theorem, we need the following lemma.

\begin{lemma}
\label{contprop}
Let $\fra$ be a non-autonomous form satisfying \eqref{bddcoer}, \eqref{measurable} 
and \eqref{omega}. Denote by ${\mathcal{A}}(t):V\to V'$ the operator associated 
with $\fra(t;.,.)$ and by $A(t)$ its part in $H$. Then for all 
$\lambda\notin \Sigma_\theta$, $s\in(0,T]$ and $\sigma>0$
the following mappings
\begin{align}
t&\mapsto (\lambda{\rm Id}-tA(0))^{-1}\in{\mathscr{L}}(V),
\label{cont1}\\[4pt]
t&\mapsto {\mathcal{A}}(t)-{\mathcal{A}}(s)\in{\mathscr{L}}(V_\gamma,V_\gamma'),
\label{cont2}\\[4pt]
t&\mapsto (\lambda{\rm Id}-{\mathcal{A}}(t))^{-1}\in{\mathscr{L}}(V_\gamma',V),
\label{cont3}\\[4pt]
t&\mapsto (\lambda{\rm Id}-t{\mathcal{A}}(t))^{-1}\in{\mathscr{L}}(V_\gamma',V),
\label{cont4}\\[4pt]
t&\mapsto e^{-\sigma {\mathcal{A}}(t)}\in{\mathscr{L}}(V_\gamma',V)
\label{cont5}
\end{align}
are continuous on $(0,T]$. 
\end{lemma}

\begin{proof}
To prove \eqref{cont1}, we write for $t,s\in(0,T]$ and $\lambda\notin\Sigma_\theta$
$$
(\lambda{\rm Id}-tA(0))^{-1}-(\lambda{\rm Id}-sA(0))^{-1}
= \Bigl(\frac{1}{s}-\frac{1}{t}\Bigr)\,tA(0)(\lambda{\rm Id}-tA(0))^{-1}
\Bigl(\frac{\lambda}{s}{\rm Id}-A(0)\Bigr)^{-1}.
$$
Thanks to \eqref{1.6} for $\ell=1$ the operators $tA(0)(\lambda{\rm Id}-tA(0))^{-1}$ and
$\Bigl(\frac{\lambda}{s}{\rm Id}-A(0)\Bigr)^{-1}$ are uniformly (w.r.t. $t$, $s$ and $\lambda$) 
bounded in $V$. Therefore
$$
(\lambda{\rm Id}-tA(0))^{-1}-(\lambda{\rm Id}-sA(0))^{-1}\longrightarrow 0 
\quad {\mbox{in }{\mathscr{L}}(V)}\mbox{ as } s\to t.
$$
The claim \eqref{cont2} follows immediately from \eqref{omega} since the 
latter implies
\begin{equation}
\label{cont2.1}
\|{\mathcal{A}}(t)-{\mathcal{A}}(s)\|_{{\mathscr{L}}(V,V_\gamma')}\le 
\omega(|t-s|)\longrightarrow 0 \quad \mbox{as } s\to t.
\end{equation}
We now prove \eqref{cont3} as follows. Let $t,s\in (0,T]$ and 
$\lambda\notin\Sigma_\theta$. We have
\begin{equation}
\label{FPS}
(\lambda{\rm Id}-{\mathcal{A}}(t))^{-1}-(\lambda{\rm Id}-{\mathcal{A}}(s))^{-1}
= (\lambda{\rm Id}-{\mathcal{A}}(t))^{-1}\bigl({\mathcal{A}}(t)-{\mathcal{A}}(s)\bigr)
(\lambda{\rm Id}-{\mathcal{A}}(s))^{-1}.
\end{equation}
Therefore, by \eqref{1.2} for $\ell=\gamma$, using \eqref{cont2.1} we have
$$
\bigl\|(\lambda{\rm Id}-{\mathcal{A}}(t))^{-1}-
(\lambda{\rm Id}-{\mathcal{A}}(s))^{-1}\bigr\|_{{\mathscr{L}}(V_\gamma',V)}
\le \frac{c\,\omega(|t-s|)}{(1+|\lambda|)^{1-\gamma}}\longrightarrow 0
\quad\mbox{as }s\to t,
$$
which proves \eqref{cont3}. The proof of \eqref{cont4} combines the ideas of the 
proofs of \eqref{cont1} and \eqref{cont3}. We write
\begin{align*}
(\lambda{\rm Id}-t{\mathcal{A}}(t))^{-1}-(\lambda{\rm Id}-s{\mathcal{A}}(s))^{-1}
=&\Bigl(\frac{1}{s}-\frac{1}{t}\Bigr)\,
A(t)\Bigl(\frac{\lambda}{t}\,{\rm Id}-A(t)\Bigr)^{-1}\Bigl(\frac{\lambda}{s}\,
{\rm Id}-{\mathcal{A}}(s)\Bigr)^{-1}
\\[4pt]
&+\frac{1}{t}\,\Bigl(\frac{\lambda}{t}\,{\rm Id}-{\mathcal{A}}(t)\Bigr)^{-1}
({\mathcal{A}}(t)-{\mathcal{A}}(s))\Bigl(\frac{\lambda}{s}\,
{\rm Id}-{\mathcal{A}}(s)\Bigr)^{-1}
\end{align*}
which implies the following estimate thanks to \eqref{1.6} for $\ell=1$, 
\eqref{1.2} for $\ell=\gamma$ and \eqref{cont2.1}
\begin{align*}
&\bigl\|(\lambda{\rm Id}-t{\mathcal{A}}(t))^{-1}-
(\lambda{\rm Id}-s{\mathcal{A}}(s))^{-1}\bigr\|_{{\mathscr{L}}(V_\gamma',V)}
\\[4pt]
\le& \frac{c}{(1+\frac{|\lambda|}{s})^{\frac{1-\gamma}{2}}}\,
\left((c+1)\Bigl|\frac{1}{s}-\frac{1}{t}\Bigr|
+\frac{1}{t}\,\frac{c\,\omega(|t-s|)}{(1+\frac{|\lambda|}{t})^{\frac{1-\gamma}{2}}}
\right)
\longrightarrow 0 \quad \mbox{as } s\to t.
\end{align*}
Finally, we show \eqref{cont5} using the representation \eqref{semigroup} for the 
semigroup and \eqref{FPS}:
$$
e^{-\sigma{\mathcal{A}}(t)}-e^{-\sigma{\mathcal{A}}(s)}
=\frac{1}{2\pi i}\int_\Gamma e^{-\lambda \sigma}
(\lambda{\rm Id}-{\mathcal{A}}(t))^{-1}
({\mathcal{A}}(t)-{\mathcal{A}}(s))(\lambda {\rm Id}-{\mathcal{A}}(s))^{-1} \, 
{\rm d}\lambda
$$
we obtain, using \eqref{1.2} for $\ell=\gamma$ and \eqref{cont2.1}
\begin{align*}
\bigl\|e^{-\sigma{\mathcal{A}}(t)}-e^{-\sigma{\mathcal{A}}(s)}
\bigr\|_{{\mathscr{L}}(V_\gamma',V)}
&\le c\,\omega(|t-s|)
\int_0^\infty e^{-\sigma r\cos\vartheta}\frac{1}{(1+r)^{1-\gamma}}\,{\rm d}r
\\[4pt]
&\longrightarrow 0\quad \mbox{as } s\to t \ \mbox{ for all }\sigma>0,
\end{align*}
which proves the claim.
\end{proof}

\begin{proof}[Proof of Theorem~\ref{MRcont}]
Assume that $u\in MR_a\subset MR(V,V')\subset{\mathscr{C}}([0,T];H)$. 
Let $f=\dot u(\cdot) +A(\cdot)u(\cdot)$
and $u_0=u(0)$: $f\in L^2(0,T;H)$, $u_0\in V$ and $u$ satisfies \eqref{CP}. By 
Proposition~\ref{TLJ}, we have $u=u_1+u_2+u_3$ where
\begin{align*}
&u_1(t)=e^{-tA(t)}u_0,\quad u_2(t)=\int_0^te^{-(t-s)A(t)}f(s)\,{\rm d}s,
\\[4pt]
&\mbox{and}
\quad u_3(t)=\int_0^te^{-(t-s){\mathcal{A}}(t)}\bigl({\mathcal{A}}(t)-
{\mathcal{A}}(s)\bigr)u(s)\,{\rm d}s.
\end{align*}
We will show that each term $u_j$, $j=1,2,3$, belongs to ${\mathscr{C}}([0,T];V)$.

\medskip

\noindent
{\tt Step 1:} We claim that $u_1\in{\mathscr{C}}([0,T];V)$. Indeed, 
$u_0\in V$ and since $({e^{-tA(0)}}_{|_V})_{t\ge 0}$ defines a $C_0$-semigroup 
one has $t\mapsto e^{-tA(0)}u_0\in{\mathscr{C}}([0,T];V)$.
Let us first consider the case where $t>0$. We have
\begin{align}
e^{-tA(t)}u_0-e^{-tA(0)}u_0
&=\frac{1}{2\pi i}\int_\Gamma e^{-t\lambda}
(\lambda{\rm Id}-{\mathcal{A}}(t))^{-1}({\mathcal{A}}(t)-{\mathcal{A}}(0))
(\lambda{\rm Id}-A(0))^{-1}u_0\,{\rm d}\lambda
\label{PMG}\\[4pt]
&=\frac{1}{2\pi i}\int_\Gamma e^{-\eta}\Bigl(\frac{\eta}{t}\,{\rm Id}-
{\mathcal{A}}(t)\Bigr)^{-1}
({\mathcal{A}}(t)-{\mathcal{A}}(0))\Bigl(\frac{\eta}{t}\,{\rm Id}-A(0)\Bigr)^{-1}
u_0\,\frac{{\rm d}\eta}{t}.
\nonumber
\end{align}
Estimates \eqref{1.6} for $\ell=1$ and \eqref{1.2} for $\ell=\gamma$ imply
\begin{align*}
&\Bigl\|\frac{1}{t}\,e^{-\eta}\Bigl(\frac{\eta}{t}\,{\rm Id}-{\mathcal{A}}(t)\Bigr)^{-1}
({\mathcal{A}}(t)-{\mathcal{A}}(0))
\Bigl(\frac{\eta}{t}\,{\rm Id}-A(0)\Bigr)^{-1}u_0\Bigr\|_V
\\[4pt]
\le&\frac{c}{t}\,e^{-\Re e\eta}\,\frac{1}{(1+\frac{\eta}{t})^{\frac{1-\gamma}{2}}}\,
\omega(t)\,\frac{1}{1+\frac{\eta}{t}}\,\|u_0\|_V
\\[4pt]
\le&\frac{c}{t}\,e^{-|\eta|\cos\vartheta}\,\frac{1}{(1+\frac{\eta}{t})^{\frac{1-\gamma}{2}}}
\,\omega(t)\,\frac{1}{(1+\frac{\eta}{t})^{1/2+\gamma/4}}\,\|u_0\|_V
\\[4pt]
\le& c\,e^{-|\eta|\cos\vartheta}\,t^{-\gamma/4}\omega(t)\,\frac{1}{|\eta|^{1-\gamma/4}}
\,\|u_0\|_V
\\[4pt]
\le& c\,\frac{e^{-|\eta|\cos\vartheta}}{|\eta|^{1-\gamma/4}}\,
\bigl(t^{-\gamma/2}\omega(t)\bigr)\,t^{\gamma/4}\,\|u_0\|_V.
\end{align*}
Since $r\mapsto \frac{e^{-r\cos\theta}}{r^{1-\gamma/4}}$ is integrable on 
$(0,\infty)$ and
$$
t\mapsto e^{-\eta t}(\eta{\rm Id}-tA(t))^{-1}(A(t)-A(0))(\eta{\rm Id}-tA(0))^{-1}u_0\in V
$$
is continuous on $(0,T]$
we may apply Lebesgue's dominated convergence theorem. 
Therefore we obtain the continuity of $t\mapsto e^{-tA(t)}u_0-e^{-tA(0)}u_0\in V$ 
on $(0,T]$. It remains to prove the continuity at $0$. Using the formula \eqref{PMG}, 
thanks to \eqref{1.6} for $\ell=1$ and \eqref{1.2} for $\ell=\gamma$,we obtain 
the following estimate
$$
\|e^{-tA(t)}u_0-e^{-tA(0)}u_0\|_V\le C\,\omega(t)\int_0^\infty
\frac{1}{(1+r)^{\frac{3-\gamma}{2}}}\,{\rm d}r
$$
where we have used that $|e^{-t\lambda}|\le 1$ for all $\lambda\in\Gamma$. Since 
$\omega(t)\longrightarrow 0$ as $t\to 0$, this proves that 
$t\mapsto e^{-tA(t)}u_0-e^{-tA(0)}u_0\in V$ is continuous on $[0,T]$, 
and ultimately that $u_1$ is continuous on $[0,T]$.

\medskip

\noindent
{\tt Step 2:} We claim that $u_2\in{\mathscr{C}}([0,T];V)$. The embedding
$$
H^1(0,T;H)\cap L^2(0,T;D(A(0)))\hookrightarrow {\mathscr{C}}([0,T];V)
$$
(recall that ${\mathfrak{a}}(0;\cdot,\cdot)$ has the square root property so that 
$V=D(A(0)^{\frac{1}{2}})$), together with the maximal regularity property in 
the autonomous case \eqref{maxreg-aut} imply that
$$
t\mapsto \int_0^t e^{-(t-s)A(0)}f(s)\,{\rm d}s\in {\mathscr{C}}([0,T];V).
$$
It suffices to prove now that
$$
\phi:t\mapsto \int_0^t e^{-(t-s)A(t)}f(s)\,{\rm d}s-\int_0^t e^{-(t-s)A(0)}f(s)\,{\rm d}s
\in {\mathscr{C}}([0,T];V).
$$
For every $t\in [0,T]$ we have
$$
\phi(t)=\int_0^t \frac{1}{2\pi i}\int_\Gamma e^{-(t-s)\lambda}
(\lambda{\rm Id}-{\mathcal{A}}(t))^{-1}({\mathcal{A}}(t)-{\mathcal{A}}(0))
(\lambda{\rm Id}-A(0))^{-1}f(s)\,{\rm d}\lambda\,{\rm d}s.
$$
This integral is convergent in $V$. Indeed, by \eqref{1.5} for $\ell=1$ and
\eqref{1.2} for $\ell=\gamma$,
\begin{align*}
&\bigl\|e^{-(t-s)\lambda}(\lambda{\rm Id}-{\mathcal{A}}(t))^{-1}
({\mathcal{A}}(t)-{\mathcal{A}}(0))(\lambda{\rm Id}-A(0))^{-1}f(s)\bigr\|_V
\\[4pt]
\le& c\,e^{-(t-s)|\lambda|\cos\vartheta}\,\frac{1}{(1+|\lambda|)^{\frac{1-\gamma}{2}}}\,
\omega(t) \frac{1}{(1+|\lambda|)^{1/2}}\,\|f(s)\|_H
\\[4pt]
\le&c\,e^{-(t-s)|\lambda|\cos\vartheta}\,\frac{\omega(t)}{(1+|\lambda|)^{1-\gamma/2}}
\,\|f(s)\|_H
\end{align*}
and the function 
$(s,r)\mapsto \frac{e^{-(t-s)r\cos\vartheta}}{(1+r)^{1-\gamma/2}}\|f(s)\|_H$
is integrable on $[0,t]\times[0,+\infty)$. We can then apply Fubini's theorem and 
obtain the following representation for $\phi$
$$
\phi(t)= \frac{1}{2\pi i}\int_\Gamma 
(\lambda{\rm Id}-A(t))^{-1}(A(t)-A(0))(\lambda{\rm Id}-A(0))^{-1}
\Bigl(\int_0^t e^{-(t-s)\lambda} f(s)\,{\rm d}s\Bigr)\,{\rm d}\lambda.
$$
The following two facts are the keys to prove the continuity of $\phi$: 
\begin{itemize}
\item[-]
$t\mapsto \int_0^t e^{-(t-s)\lambda} f(s)\,{\rm d}s\in {\mathscr{C}}([0,T];H)$ and for 
all $\lambda\in \Gamma\setminus\{0\}$,
$$
\Bigl\|\int_0^t e^{-(t-s)\lambda} f(s)\,{\rm d}s\Bigr\|_H
\le \frac{1}{\sqrt{2|\lambda|\cos\vartheta}}\,\|f\|_{L^2(0,T;H)};
$$
\item[-]
$t\mapsto(\lambda{\rm Id}-A(t))^{-1}(A(t)-A(0))(\lambda{\rm Id}-A(0))^{-1}\in 
{\mathscr{C}}([0,T];{\mathscr{L}}(H,V))$ thanks to \eqref{cont2} and \eqref{cont3}
and for all $\lambda\in\Gamma$,
$$
\bigl\|(\lambda{\rm Id}-A(t))^{-1}
(A(t)-A(0))(\lambda{\rm Id}-A(0))^{-1}\bigr\|_{{\mathscr{L}}(H,V)}
\le \frac{c\,\omega(t)}{(1+|\lambda|)^{1-\gamma/2}}.
$$
\end{itemize}
Since $r\mapsto \frac{1}{\sqrt{r}(1+r)^{1-\gamma/2}}\in L^1(0,\infty)$ we can apply 
Lebesgue's dominated convergence theorem and we obtain that 
$\phi\in {\mathscr{C}}([0,T];V)$ and
$$
\|\phi(t)\|_V\le c\,\omega(t)\Bigl(\int_0^\infty \frac{1}{\sqrt{r}(1+r)^{1-\gamma/2}}\,
{\rm d}r\Bigr)\|f\|_{L^2(0,T;H)},
$$
which proves the claim.

\medskip

\noindent
{\tt Step 3:}
We conclude in a similar way as in Step~3 of the proof of 
Theorem~\ref{maxreg-nonaut}. We define for $h\in{\mathscr{C}}([0,T];V)$
\begin{align*}
(Ph)(t):=&\int_0^te^{-(t-s)A(t)}({\mathcal{A}}(t)-{\mathcal{A}}(s))h(s)\,{\rm d}s
\\[4pt]
=&A(t)^{-1/2}\int_0^tA(t)^{1/2}e^{-(t-s)A(t)}
({\mathcal{A}}(t)-{\mathcal{A}}(s))h(s)\,{\rm d}s.
\end{align*}
Thanks to \eqref{cont2} and \eqref{cont5}, $Ph$ is continuous
on $[0,T]$ with values in $V$ by Lebesgue's dominated convergence theorem.
For all $t,s\in[0,T]$, $t>s$, the estimate \eqref{1.12} for $\ell=\gamma$,
the fact that $(e^{-\sigma A(t)})_{\sigma\ge0}$ is a holomorphic semigroup
in $V$ and the property \eqref{omega} give
$$
\bigl\|A(t)^{1/2}e^{-(t-s){\mathcal{A}}(t)}({\mathcal{A}}(t)-{\mathcal{A}}(s))h(s)\bigr\|_V
\le \frac{c\omega(t-s)}{(t-s)^{1+\gamma/2}}\,\|h(s)\|_V.
$$
Let $\varepsilon>0$. As in Step~3 of the proof of Theorem~\ref{maxreg-nonaut},
Replacing $A(t)$ by $A(t)+\mu{\rm Id}\,$ we may assume that
$\|A(t)^{-1/2}\|_{{\mathscr{L}}(V)}\le\varepsilon$. Therefore, we have
$$
\sup_{t\in[0,T]}\|(Ph)(t)\|_V\le 
c\varepsilon\Bigl(\int_0^T\frac{\omega(\sigma)}{\sigma^{1+\gamma/2}}
\,{\rm d}\sigma\Bigr)\sup_{t\in[0,T]}\|h(t)\|_V.
$$
Now, thanks to \eqref{omegaL1}, choosing $\varepsilon$ small enough we 
can arrange that $\|P\|_{{\mathscr{L}}({\mathscr{C}}([0,T];V))}\le \frac{1}{2}$. 
Thus ${\rm Id}\,-P$ is invertible in ${\mathscr{L}}\bigl({\mathscr{C}}([0,T];V)\bigr)$.
We have, by definition of $u_1$ and $u_2$, that $u-Pu=u_1+u_2$. 
Since we have proved in Step~1 and Step~2 that
$u_1+u_2\in {\mathscr{C}}([0,T];V)$, it shows that
$u=({\rm Id}\,-P)^{-1}(u_1+u_2)\in{\mathscr{C}}([0,T];V)$. 
\end{proof}


\section{Non-autonomous Robin boundary conditions}
\label{section5}

Let $\Omega\subset {\mathbb{R}}^N$ be a bounded open set with Lipschitz 
boundary. We denote by $\partial\Omega$ the boundary of $\Omega$ and take 
$L^2(\partial\Omega)$ with respect to the $(N-1)$-dimensional Hausdorff measure. 
There exists a unique bounded operator 
${\rm Tr}:H^1(\Omega)\to L^2(\partial\Omega)$ such that
${\rm Tr}(u)=u_{|_{\partial\Omega}}$ if 
$u\in H^1(\Omega)\cap{\mathscr{C}}(\overline{\Omega})$. We call ${\rm Tr}(u)$ 
the {\em trace} of $u$ and also use the notation $u_{|_{\partial\Omega}}$ for 
$u\in H^1(\Omega)$. Let $\alpha>\frac{1}{4}$ and 
$B:[0,T]\to{\mathscr{L}}(L^2(\partial\Omega))$ be a mapping such that
\begin{equation}
\label{Bcont}
\|B(t)-B(s)\|_{{\mathscr{L}}(L^2(\partial\Omega))}\le c|t-s|^\alpha
\end{equation}
for all $t,s\in[0,T]$ and some $c\ge 0$.
We need some further definitions. If $u\in H^1(\Omega)$ such that 
$\Delta u\in L^2(\Omega)$ and if $b\in L^2(\partial\Omega)$ then we write
$$
\partial_\nu u=b\quad \mbox{if}\quad 
\int_\Omega\Delta u\,\overline{v} +\int_\Omega\nabla u\cdot\overline{\nabla v} 
=\int_{\partial\Omega} b\overline{v},\ 
\mbox{ for all }v\in H^1(\Omega).
$$
This means that we define the {\em normal derivative} $\partial_\nu u$ of $u$
by the validity of Green's formula. Now we can formulate our main result on the 
heat equation with non-autonomous Robin boundary conditions.

\begin{theorem}
\label{maxregRobin}
Let $H=L^2(\Omega)$, $f\in L^2(0,T;H)$, $u_0\in H^1(\Omega)$. Then there 
exists a unique function $u\in H^1(0,T;H)\cap L^2(0,T;H^1(\Omega))$ such that 
$\Delta u\in L^2(0,T;H)$ and 
$$
\left\{\begin{array}{rcll}
\dot u(t)-\Delta u(t)&=&f(t)\quad&t-a.e.
\\[4pt]
\partial_\nu u(t)+B(t)u(t)_{|_{\partial\Omega}}&=&0&t-a.e.
\\[4pt]
u(0)&=&u_0.&
\end{array}\right.
$$
\end{theorem}

\begin{proof}
Given is $\alpha>\frac{1}{4}$. Central for the proof is a result by Jerison and Kenig
(see \cite[p.\,165]{JeKe95}; see also 
\cite[Theorem~1, Ch. V.1, \S1.1, p.\,103]{JW84}) which says 
that for $0<s<1$ there is a unique bounded linear operator
$$
{\rm Tr}_s:H^{s+1/2}(\Omega)\to H^s(\partial\Omega)
$$
such that ${\rm Tr}_s(u)=u_{|_{\partial\Omega}}$ for all 
$u\in H^{s+1/2}(\Omega)\cap {\mathscr{C}}(\overline{\Omega})$.
In particular, ${\rm Tr}_{1/2}={\rm Tr}$. Moreover 
$H^{1/2}(\partial\Omega)={\rm Tr}\bigl(H^1(\Omega)\bigr)$.
Now choose $0<s<\frac{1}{2}$ such that $\gamma:=s+\frac{1}{2}<2\alpha$. 
Then $\gamma<1$ and
$\alpha>\frac{\gamma}{2}$ as needed in Section~\ref{section3} for 
$\omega(t)=c\,t^\alpha$.
Moreover, for $u\in H^1(\Omega)$ we have
\begin{align*}
\bigl\|B(t)u_{|_{\partial\Omega}}-B(s)u_{|_{\partial\Omega}}\bigr\|_{L^2
(\partial\Omega)}
&\le c\,|t-s|^\alpha\bigl\|u_{|_{\partial\Omega}}\bigr\|_{L^2(\partial\Omega)}
\\[4pt]
&\le c\,|t-s|^\alpha\bigl\|u_{|_{\partial\Omega}}\bigr\|_{H^s(\partial\Omega)}
\\[4pt]
&\le c\,|t-s|^\alpha\|u\|_{H^{s+1/2}(\Omega)}.
\end{align*}
Thus the form
$$
\fra(t;u,v):=\int_\Omega\nabla u\cdot\overline{\nabla v} +
\int_{\partial\Omega}B(t)u_{|_{\partial\Omega}}\,\overline{v}_{|_{\partial\Omega}}
$$
defined on $[0,T]\times H^1(\Omega)\times H^1(\Omega)$ satisfies condition 
\eqref{omega}.

We now choose $\mu>\|B(\cdot)\|_{L^\infty(0,T;{\mathscr{L}}(L^2(\partial\Omega))}$
so that the form 
$$
[0,T]\times H^1(\Omega)\times H^1(\Omega)\ni (t,u,v)\mapsto \fra(t;u,v)
+\mu\int_\Omega u\,\overline{v}
$$
satisfies \eqref{bddcoer}, \eqref{measurable} in addition to \eqref{omega}.
By Theorem~\ref{maxreg-nonaut} this perturbed form has maximal regularity in $H$.
It follows from Proposition~\ref{shift-inv} that also the form $\fra$ has maximal 
regularity in $H$. Denote by $A(t)$ the operator associated with 
$\fra(t;\cdot,\cdot)$ in $H=L^2(\Omega)$. Then
\begin{align*}
D(A(t))&=\bigl\{u\in H^1(\Omega): \Delta u\in L^2(\Omega), 
\partial_\nu u+B(t)u_{|_{\partial\Omega}}=0\bigr\}
\\[4pt]
A(t)u&=-\Delta u,
\end{align*}
as is easy to see using the definition of $\partial_\nu u$ by Green's formula. 
Thus maximal regularity in $H$ is exactly the statement of 
Theorem~\ref{maxregRobin}.
\end{proof}

\noindent
Next we consider a non-linear problem. Keeping the assumptions and settings 
of this section we consider bounded continuous functions 
$\beta_j:{\mathbb{R}}\to{\mathbb{R}}$, $j=0,1,...,N$.

\begin{theorem}
\label{aioli}
Let $u_0\in H^1(\Omega)$, $f\in L^2(0,T;L^2(\Omega)$. Then there exists 
$u\in H^1(0,T;L^2(\Omega))\cap L^2(0,T;H^1(\Omega))$ such that 
$\Delta u\in L^2(0,T;L^2(\Omega))$ and
\begin{equation}
\label{aioli1}
\left\{\begin{array}{rcl}
\displaystyle{\dot u(t)-\Delta u(t) +\sum_{j=1}^N\beta_j(u(t))\partial_ju(t)
+\beta_0(u(t))u(t)}&=&f(t) \quad \mbox{a.e. on }\Omega
\\[4pt]
\partial_\nu u(t)+B(t)u(t)_{|_{\partial\Omega}}&=&0 \quad \mbox{a.e. on }
\partial\Omega
\\[4pt]
u(0)&=&u_0.
\end{array}\right.
\end{equation}
\end{theorem}

\begin{proof}
We let $\fra(t;.,.):H^1(\Omega)\times H^1(\Omega)\to{\mathbb{C}}$  
be defined as before. Given $w\in L^2(0,T;L^2(\Omega))$ we define 
the form 
$\fra_2^w:[0,T]\times H^1(\Omega)\times H^1(\Omega)\to{\mathbb{C}}$ by
$$
\fra_2^w(t;u,v)=\int_\Omega\sum_{j=1}^N\beta_j(w(t))\partial_ju \,v
+\int_\Omega \beta_0(w(t))u\,v, \quad u\in H^1(\Omega), v\in L^2(\Omega).
$$
Then ${\mathfrak{a}}+{\mathfrak{a}}_2^w$ satisfies \eqref{omega} with 
$\gamma=2\alpha$ and with a constant $c$ which does not depend on 
$w\in L^2(0,T;L^2(\Omega))$. Thus by 
Theorem~\ref{maxreg-nonaut}, there exists a unique solution 
$u$ belonging to the space 
$$
E:=H^1(0,T;L^2(\Omega))\cap L^2(0,T;H^1(\Omega))
$$ 
with $\Delta u\in L^2(0,T;L^2(\Omega))$ of the problem
$$
\left\{\begin{array}{rcl}
\displaystyle{\dot u(t)-\Delta u(t) +\sum_{j=1}^N\beta_j(w(t))\partial_ju(t)+
\beta_0(w(t))u(t)}
&=&f(t) \quad \mbox{a.e. on }\Omega
\\[4pt]
\partial_\nu u(t)+B(t)u(t)_{|_{\partial\Omega}}&=&0 \quad \mbox{a.e. on }
\partial\Omega
\\[4pt]
u(0)&=&u_0.
\end{array}\right.
$$
We define $Tw:=u$. Then $T:L^2(0,T;L^2(\Omega))\to L^2(0,T;L^2(\Omega))$ is 
a continuous mapping (as is easy to see). Moreover, $TL^2(0,T;L^2(\Omega))$ is 
a bounded subset of $E$. This follows from Theorem~\ref{maxreg-nonaut}. 
Since the embedding of $H^1(\Omega)$ into $L^2(\Omega)$ is compact (recall that 
$\Omega$ is bounded), it follows from the Lemma of Aubin-Lions that the 
embedding of $E$ into $L^2(0,T;L^2(\Omega))$ is compact as well. It follows from 
Schauder's Fixed Point Theorem that $T$ has a fixed point $u$. 
This function $u$ solves the problem.
\end{proof}

\section*{Aknowledgements}

We are most grateful to El Maati Ouhabaz and Dominik Dier for fruitful and
enjoyable discussions. We would like to thank the anonymous referee who helped
to improve Section~4.


{\small
\bibliographystyle{amsplain}

}
\end{document}